\documentclass[leqno,12pt]{amsart} 

\usepackage[utf8x]{inputenc} 
\usepackage{amsmath, nicefrac, a4wide, mathtools}
\usepackage{amssymb}
\usepackage{amsfonts}                       
\usepackage{enumerate}

\usepackage{hyperref}

\theoremstyle{plain} 
\newtheorem{theorem}{\indent\sc Theorem}[section]
\newtheorem{lemma}[theorem]{\indent\sc Lemma}
\newtheorem{corollary}[theorem]{\indent\sc Corollary}
\newtheorem{proposition}[theorem]{\indent\sc Proposition}

\theoremstyle{definition} 
\newtheorem{Definition}[theorem]{\indent\sc Definition}
\newtheorem{remark}[theorem]{\indent\sc Remark}

\newcommand{\LieSO}{{\mathrm{SO}}}

\newcommand{\lieA}{{\mathfrak{a}}}

\newcommand{\lieG}{{\mathfrak g}}
\newcommand{\lieH}{{\mathfrak h}}
\newcommand{\lieK}{{\mathfrak k}}

\newcommand{\lieP}{{\mathfrak p}}

\newcommand{\lieZ}{{\mathfrak z}}

\newcommand{\lieSO}{{\mathfrak{so}}}

\newcommand{\Ad}{{\mathrm{Ad}}}

\newcommand{\ad}{{\mathrm{ad}}}

\newcommand{\Exp}{{\mathrm{Exp}}}

\newcommand{\Root}{{\mathcal{R}}}

\DeclareMathOperator{\aff}{aff}

\def\R{{\mathbb R}}
\def\Sphere{{S}}
\def\N{{\mathbb N}}

\def\Z{{\mathbb Z}}
\def\C{{\mathbb C}}

\title[Submanifolds of Clifford type]
{Extrinsically Symmetric Spaces, Submanifolds of Clifford Type and a Theorem of Harish-Chandra}
\author[Eschenburg]{Jost-Hinrich Eschenburg}
\email{jost-hinrich.eschenburg@math.uni-augsburg.de}
\author[Heintze]{Ernst Heintze}
\email{ernst.heintze@math.uni-augsburg.de}
\author[Quast]{Peter Quast}
\email{peter.quast@math.uni-augsburg.de}
\address{Institut f\"ur Mathematik, Universit\"at Augsburg, 86135 Augsburg, Germany}

\date{\today}

\subjclass[2020]{53C35, 53C40}

\keywords{extrinsically symmetric spaces, Clifford tori, Clifford type, strongly orthogonal roots}
\dedicatory{To the memory of Joseph A.\ Wolf.}

\begin{document}

\begin{abstract}
We prove that a compact,  intrinsically symmetric submanifold of a Euclidean space is extrinsically symmetric if and only if its maximal tori are Clifford tori in the ambient space. Moreover,  we show that this result can be used to give a geometric proof of a result of Harish-Chandra on strongly orthogonal roots in semisimple Lie algebras.
\end{abstract}

\maketitle
\section*{Introduction}

A closed submanifold $X$ of a Euclidean space $E$ is called of \emph{Clifford type} if every geodesic in $X$ is contained in a totally geodesic submanifold of $X$ which is a Clifford torus in $E.$ For example, $X$ is of Clifford type if all its geodesics are planar circles in $E.$
It is easy to see (Proposition \ref{PROP: Clifford type implies extrinsic symmetry}) that submanifolds of Clifford type are compact \emph{extrinsically symmetric spaces}, that is they are compact submanifolds which are invariant under the orthogonal reflections along their affine normal spaces. By Theorem \ref{THM: Maximal tori of ess are Clifford} also the converse holds. Surprisingly, this result is closely related to a theorem of \textsc{Harish-Chandra} on the existence of strongly orthogonal roots of noncompact type in semisimple Lie algebras. We
show in fact how \textsc{Harish-Chandra}'s result can be deduced from Theorem \ref{THM: Maximal tori of ess are Clifford} by an explicit geometric construction and we indicate also a proof of the opposite direction.

\section{Clifford tori}
\label{SECT: Clifford tori}

Let $r\geq 0$ be an integer and $\rho_1,\dots,\rho_r$ be positive real numbers. We call
$$C:=C_{\rho_1\cdots \rho_r}=\{z=(z_1,\dots, z_r)\in\C^r:\; |z_j|=\rho_j\; \text{for all}\; j=1,\dots, r\}$$
a \emph{standard Clifford torus}, generalizing the classical case $r=2$ and $\rho_1=\rho_2=1.$ If $r=0$ we have $C=\{0\}.$ The inner product
$\langle z,w\rangle=\mathrm{Re}\Big(\sum\limits_{j=1}^r z_j\overline{w_j}\Big)$ turns $\C^r$ into a Euclidean space and $C$ becomes an extrinsically symmetric submanifold of $\C^r$ as round circles in planes are extrinsically symmetric and $C$ in $\C^r$
is a product of such submanifolds. For each $w=(w_1,\dots, w_r)\in C$ the circles
$$C_k(w)=\{z\in\C^r:\; |z_k|=\rho_k\; \text{and}\; z_j=w_j\; \text{for all}\; k\neq j\},$$
$k=1,\dots r,$
are called \emph{generating circles} of $C$ at $w.$ They are contained in the pairwise orthogonal affine planes
$\C^r_k(w):=\{z\in\C^r:\; z_j=w_j\; \text{for all}\; k\neq j\}$ and $C$ is isometric to their product. \par

More generally, a submanifold $C$ of a Euclidean space $E$ is called a \emph{Clifford torus} if it is the image of a standard
Clifford torus $C'$ in $\C^r$ under an affine isometric (not necessarily surjective) map $\varphi:\C^r\to E.$
Fixing $\varphi$ we have for each $p=\varphi(w)\in C$ the generating circles $C_k(p):=\varphi(C'_k(w))$ for $k=1,\dots, r$ lying in pairwise orthogonal affine planes
$\varphi(\C^r_k(w))$.
Up to order these generating circles of $C$ are independent of $\varphi$ as they only depend on the inner geometry of $C.$ In fact, the unit lattice
$\Gamma=\{v\in T_pC:\; \Exp_p(v)=p\},$ where $\Exp_p$ denotes the Riemannian exponential map of $C$ at $p,$ is rectangular, that is
$\Gamma$ has an orthogonal basis. Indeed, $\Gamma$ is the direct sum of the unit lattices of the generating circles $C_k(p),\; k=1\dots r.$
Thus uniqueness of the generating circles up to order follows from:

\begin{lemma}\label{LEM: rect basis unit lattice unique}
 Let $\Gamma$ be a rectangular lattice in a Euclidean vector space $V.$ Then an orthogonal basis of $\Gamma$ is unique up to signs and order.
\end{lemma}

\begin{proof}
 Let $B=\{b_1,\dots, b_r\}$ be an orthogonal basis of $\Gamma,$ that is
 $\Gamma=\big\{\sum\limits_{j=1}^r\alpha_j b_j:\; \alpha_j\in\Z\; \text{for all}\; j=1,\dots, r\big\}.$ Assuming that $b_1\dots, b_s,\; s\leq r,$ are the shortest elements in $B$ then
 $\pm b_1,\dots, \pm b_s$ are the shortest element in $\Gamma.$ Thus $b_1,\dots, b_s\in\Gamma$ are unique up to sign and order. The Lemma
 now follows by induction on the dimension of $V.$
\end{proof}

As the standard Clifford torus $C'$ is extrinsically symmetric in $\C^r,$ the Clifford torus  $C=\varphi(C')$ is extrinsically symmetric
in $\varphi(\C^r)\subset E,$ which is the affine hull $\aff(C)$ of $C,$ and thus extrinsically symmetric in $E.$ Since at each $p\in C$ the affine hull $\aff(C)$ splits orthogonally as the direct sum $\bigoplus_{k=1}^r \aff(C_k(p))$ we have:

\begin{lemma}\label{LEM: Reflections on generating circles}
Let $p\in C$ and let $\psi$ be an isometry of $E$ with $\psi(p)=p.$
If $\psi$ induces a reflection on each generating circle $C_k(p)$ then $\psi$ leaves $C$ invariant.
\end{lemma}

\begin{lemma}\label{LEM: Fixed point sets of isometries Clifford tori}
Each connected component of the fixed point set of an isometry $f$ of $C$ is a Clifford torus in $E.$
\end{lemma}
\begin{proof}
Let $p\in C$ be fixed by $f$ and let $\Gamma\subset T_pC$ be the unit lattice of $C$ at $p.$
By Lemma \ref{LEM: rect basis unit lattice unique} there exists an orthogonal basis $b_1,\dots, b_r$ of $T_pC$ unique up to sign and permutations with $\Gamma=\mathrm{span}_{\Z}\{b_1,\dots, b_r\}.$
The differential $f_*$ of $f$ at $p$ is an orthogonal map of $T_pC$ which preserves $\Gamma_p$ and therefore acts on $\{\R b_1,\dots, \R b_r\}$
by permutation. We decompose $I\coloneqq \{1,...,r\}$ as $I=\bigsqcup\limits_{j=1}^k I_j$ into non-empty subsets $I_j$ such that $f_*$ acts
cyclically on the sets $\{\R b_s: s\in I_j\},\; j\in\{1,\dots, k\}.$ Then
$C$ is the product of the Clifford tori $T_j=\Exp_p(\mathrm{span}_{\R}\{b_s: s\in I_j\}),\; j=1,\dots, k,$ which are contained in pairwise
orthogonal affine subspaces and are invariant under $f.$
It therefore suffices to assume that $f_*$ acts cyclically on $\{\R b_1,\dots, \R b_r\}$
and
$f_*(b_j)=b_{j+1}$ for all $j\in\{1,\dots, r-1\}.$ Thus $f_*(b_r)\in\{-b_1, b_1\}.$
If $f_*(b_r)=-b_1,$ then $f_*$ has no fixed vector and $p$ is an isolated fixed point of $f.$ If $f_*(b_r)=b_1,$ then the fixed point set of
$f_*$ is $\R(b_1+\dots +b_r)$ and the connected component of the fixed point set of $f$ through $p$ is
the planar circle $\Exp_p(\R(b_1+\dots +b_r)).$
\end{proof}
\section{Submanifolds of Clifford type}
\label{SECT: Submanifolds of Clifford type}

\begin{Definition}
An (embedded) submanifold $X$ of a Euclidean space $E$ is called \emph{of Clifford type}
if every geodesic of $X$ lies in a totally geodesic submanifold $C$ of $X$
which is a Clifford torus in $E$.
\end{Definition}

For brevity we call a totally geodesic submanifold in $X\subset E$ which is a Clifford torus in $E$ a 
\emph{totally geodesic Clifford torus} in $X.$

\begin{proposition}
\label{PROP: Clifford type implies extrinsic symmetry}
 A connected submanifold $X\subset E$ of Clifford type is a compact extrinsically symmetric space.
\end{proposition}

\begin{proof} Since Clifford tori are compact $X$ is geodesically complete.
As there are no
(distance minimizing) geodesic rays in a torus $X$ is compact.\par
Let $p\in X$ and let $r_p:E\to E$ be the reflection along the affine normal space $p+N_pX$ 
of $X$ at $p$. Since any point of $X$ can be joined by a geodesic to $p$ such a geodesic lies in a totally geodesic Clifford torus $C\subset X$ containing $p.$ In order to prove $r_p(X)=X$ it suffices to show that $r_p(C)=C.$ 
In view of Lemma \ref{LEM: Reflections on generating circles} we have to see that $r_p$ induces the geodesic symmetry at $p$ on each generating circle $C_k(p)$ for $k=1,\dots, \dim(C).$ Now the second derivative of a geodesic parameterizing $C_k(p)$ is normal to $X.$ Thus $r_p$ leaves $\aff(C_k(p))$ invariant and induces on $C_k(p)$ the geodesic symmetry at $p.$
\end{proof}

As tori contain dense geodesics we conclude:

\begin{corollary}\label{COR: max tori are Clifford}
 Maximal tori of connected submanifolds of Clifford type are Clifford tori.
\end{corollary}

An advantage of the notion of submanifolds of Clifford type is its simple definition. It is sometimes possible to carry over results about Clifford tori to submanifolds of Clifford type. Proposition \ref{PROP: fixed sets isometries Clifford type} illustrates this principle
with Lemma \ref{LEM: Fixed point sets of isometries Clifford tori}. Recall that a non-empty connected component of an involutive isometry
of a Riemannian manifold is called a \emph{reflective subspace}.

\begin{proposition}\label{PROP: fixed sets isometries Clifford type}
 A reflective subspace $Y$ of a submanifold $X\subset E$ of Clifford type is itself of Clifford type.
\end{proposition}

\begin{proof}
By Proposition \ref{PROP: Clifford type implies extrinsic symmetry}, $X$ is a compact extrinsically symmetric submanifold of $E$ and
$Y$ is a compact symmetric space being a totally geodesic closed submanifold of $X.$ It suffices to show that any maximal torus
$T_Y$ of $Y$ is a Clifford torus. We may enlarge $T_Y$ to a maximal torus $T_X$ of $X.$ Then $T_X$ is a Clifford torus in $E$ by
Corollary \ref{COR: max tori are Clifford}. Let $f$ be an involutive isometry of $X$ having $Y$ as a component of its fixed point set.
It now suffices to show that $T_X$ is $f$-invariant, since
then $T_Y$ is a connected component of the fixed point set of $f|_{T_X}$ and
thus a Clifford torus by Lemma \ref{LEM: Fixed point sets of isometries Clifford tori}.
Fixing $p\in T_Y$ it is actually  sufficient to see that the differential $f_*$ of $f$ at $p$ leaves the tangent space of $T_X$ at $p$ invariant.\par
Let $G$ be the (full) isometry group of $X$ and $\lieG$ its Lie algebra.
Every $g\in G$ induces an automorphism of $G$ given by conjugation with $g.$ Its differential $\Ad(g)$ at the identity
is an automorphism of $\lieG.$
The involution $\Ad(s_p),$ where $s_p$ is the geodesic symmetry of $X$ at $p,$
gives rise to a splitting $\lieG=\lieK\oplus\lieP$ where $\lieK$ and $\lieP$ the $(\pm 1)$-eigenspace of
$\Ad(s_p).$ The restriction to $\lieP$ of the differential $\mathrm{pr}_*$ of $\mathrm{pr}: G\to X,\; g\mapsto g.p:=g(p)$ at the identity  identifies
$\lieP$ with $T_pX.$ Under this identification $f_*:T_pX\to T_pX$ corresponds to $\Ad(f)|_{\lieP}.$\par
Let $\lieG_{\pm}$ be the  $(\pm 1)$-eigenspace of the involution $\Ad(f).$ Then $\lieG=\lieG_+\oplus\lieG_-.$
As $s_p = s_{f(p)} = fs_pf^{-1},$ the involutions $\Ad(f)$ and $\Ad(s_p)$ commute and we get a finer decomposition
$\lieG=\lieK_+\oplus\lieK_-\oplus\lieP_+\oplus\lieP_-$ where $\lieK_{\pm}=\lieK\cap\lieG_{\pm}$ and $\lieP_{\pm}=\lieP\cap \lieG_{\pm}.$
Note that
$\lieP_+$ is identified with $T_pY.$ Let $\lieA_+\subset\lieP_+$ and $\lieA\subset\lieP$ be the maximal
abelian subspaces corresponding to the tangent spaces of $T_Y$ and $T_X$ at $p.$ Obviously
$\lieA_+\subset\lieA.$ To show that $\lieA=\lieA_+\oplus (\lieA\cap\lieP_-)$
let $\zeta\in\lieA$ and $\zeta=\zeta_++\zeta_-$ with $\zeta_{\pm}\in\lieP_{\pm}.$
Then $0=[\zeta,H]=[\zeta_+,H]+[\zeta_-,H]$ for any $H\in\lieA_+.$
From $[\lieP_{\pm},\lieP_+]\subset\lieK_{\pm}$ we obtain $[\zeta_\pm,\lieA_+]=\{0\}$. By maximality we get
$\zeta_+\in\lieA_+$  and hence $\zeta_-\in\lieA\cap\lieP_-$.
Finally, since both $\lieA_+$ and $\lieA\cap\lieP_-$ are $\Ad(f)$-invariant, $\lieA$ is $\Ad(f)$-invariant and the tangent space of $T_X$ at $p$
is $f_*$-invariant.
\end{proof}

Next we show the converse of Proposition
\ref{PROP: Clifford type implies extrinsic symmetry}.

\begin{theorem} \label{THM: Maximal tori of ess are Clifford}
Every compact extrinsically symmetric submanifold $X$ of a Euclidean space $E$ is of Clifford type. More precisely,
every maximal torus of $X$ is a Clifford torus in $E.$
\end{theorem}

To prepare the proof of Theorem \ref{THM: Maximal tori of ess are Clifford}
let $X\subset E$ be a compact, connected, extrinsically symmetric space. We may assume that the ambient Euclidean space
is a (real) Euclidean vector space
by choosing the barycenter of $X$ as origin. We therefore will denote the ambient space by $V$ instead of $E.$
Then the identity component $\hat{G}$ of the group generated by the reflections along the affine normal spaces of $X$
is a subgroup $\hat{G}\subset\LieSO(V).$
The restriction map $g\mapsto g|_X$ identifies  $\hat{G}$ with the identity component of the isometry group of $X.$
Let $p\in X.$ Since the reflection $r_p$ along the affine normal space of $X$ at $p$
restricts to the geodesic symmetry $s_p$ of $X$ at $p,$ conjugation with $r_p$ gives rise to an involutive automorphism of $\hat{G}$
whose differential at the identity coincides with $\Ad(s_p)$ on the Lie algebra of $\lieG$ of the isometry group of $X.$
Let $T_X$ be a maximal torus of $X$ that contains $p$
and let $\lieA$ be
the maximal abelian subspace of $\mathrm{Fix}(-\Ad(s_p))=:\lieP\cong T_pX$ which corresponds to the tangent space of $T_X$ at $p.$ Then
$T_X=\exp(\lieA).p$.\par

We next show an auxiliary lemma needed in the proof of Theorem \ref{THM: Maximal tori of ess are Clifford}.
Let $\gamma$ be a geodesic in $X.$ Then $\gamma''$ is a vector field along $\gamma$ normal to $X,$ more precisely $\gamma''=\alpha(\gamma',\gamma'),$ where $\alpha$ denotes the second fundamental form of $X.$
Since $X$ is extrinsically symmetric,  $\alpha$  is parallel (as $D\alpha$ is equivariant with respect to the differentials of reflections along the affine normal spaces).
Since $\gamma'$ is parallel we get
$$\gamma'''=-S_{\alpha(\gamma',\gamma')}\gamma',$$
where $S$ is the shape operator of $X.$ In particular, $\gamma'''$ is a vector field along $\gamma$ which is tangent to $X.$ 

\begin{lemma}\label{LEM: Hilfslemma}
 If $\gamma$ is a geodesic in $T_X$ with $\gamma(0)=p,$ then
 $\gamma'''(0)$ is tangent to $T_X.$
\end{lemma}

\begin{proof}
Let $\eta\in\lieK$ and $k_t=\exp(t\eta),$ $t\in\R,$ the corresponding one-parameter subgroup of $K.$ Let $v\in\lieA.$ Since $\alpha$ is invariant under extrinsic isometries of $X$ the map
$t\mapsto \|\alpha(k_t.v,k_t.v)\|^2$ is constant.
Differentiation at $t=0$ yields
 $0= \langle \alpha(v,v),\alpha(v,[\eta, v])\rangle.$
 If $v$ is a regular vector in $\lieA\cong T_pT_X,$ then $[\lieK, v]=\lieA^\perp,$ where $\lieA^\perp$ denotes the orthogonal
 complement of $\lieA$ in $\lieP$ (polarity of the linear isotropy representation). Therefore
$\langle S_{\alpha(v,v)}v, \lieA^\perp\rangle=\langle \alpha(v,v),\alpha(v,\lieA^\perp)\rangle=0$
 for all regular $v\in\lieA$ and,  by continuity, for all $v\in\lieA.$ This shows $\gamma_v'''(0)=-S_{\alpha(v,v)}v\in\lieA$ for all $v\in\lieA.$
\end{proof}

\begin{proof}[\textsc{Proof of Theorem \ref{THM: Maximal tori of ess are Clifford}}]
Since all $H\in\lieA$ are simultaneously diagonalizable over $\C$,
they can be brought  simultaneously into block diagonal form over $\R.$
More precisely, $V = V_0 \oplus V_1 \oplus \dots \oplus V_k$ such that
each $H\in \lieA$ acts trivially on $V_0$, and on each $V_j \cong \C$ ($j=1,\dots,k$)  it acts by the factor
$i\alpha_j(H)$ for some real-valued linear form $\alpha_j$ on $\lieA$.
This shows that the maximal torus $T_X = \{\exp(H).p: H\in\lieA\}$ is contained in a
Clifford torus $C$ in $V$, and each geodesic $\gamma$ of $T_X$ emanating from $p$,
that is $\gamma(t) =\exp (tH).p$, $H\in\lieA$, is a geodesic of $C$ too.
We may assume that $C$ is a Clifford torus of minimal dimension
with these properties and further that
$C = \{(z_1,\dots,z_k): |z_j| = r_j\}$ is a standard torus with $r_j>0$
and
$p = (r_1,\dots,r_k).$ Let $\gamma$ be a dense
geodesic in $T_X$ with $\gamma(0) = p$. Since $\gamma$ is also a geodesic
in $C$, it has the form $\gamma(t) =
\left(r_1e^{ia_1t},\dots,r_ke^{ia_kt}\right)$ for some $a_j\in\R$.
As $\gamma$ is dense in $T_X$ and $C$ is of minimal dimension,
the numbers $(a_j)^2$ are pairwise different and do not vanish. In fact,
$\left\{\left(re^{iat},r'e^{iat}\right): t\in \R\right\}$ and $\left\{\left(re^{iat},r'e^{-iat}\right): t\in \R\right\}$
are planar circles for any $r,r'>0$ which would allow to reduce
the dimension of $C$ if $a_{j_1}^2 = a_{j_2}^2$ for some $j_1\neq j_2.$ Similarly, if $a_j=0$, the
corresponding $\Sphere^1$-factor of $C$ could be deleted. Thus we
may assume $0 < |a_1| <\cdots < |a_k|$.\par

By construction we have $T_X\subset C$. To show equality (and thus that
$T_X$ is a Clifford torus) let $\gamma_\ell$ be the geodesic of $T_X$ emanating from $p$ which is
defined inductively by $\gamma_1 = \gamma$ and $\gamma_{\ell+1}$
with initial vector $\gamma_{\ell+1}'(0) := -\gamma_\ell'''(0)\in T_pT_X$
(see Lemma \ref{LEM: Hilfslemma}).
From $\gamma_1(t) = \left(r_1e^{ia_1t},\dots,r_ke^{ia_kt}\right)$ we get
$\gamma_\ell(t) = \Big(r_1e^{ita_1^{(3^\ell)}},\dots,r_ke^{ita_k^{(3^\ell)}}\Big)$.
Now, if $w = i(w_1.\dots,w_k) \in i\R^k = T_pC$ is perpendicular to $T_pT_X\subset T_pC$,
we get $\langle \gamma'_{\ell}(0),w\rangle = - \sum\limits_{j=1}^k r_ja_j^{(3^\ell)}w_j = 0$ for all $\ell\in\N$ and therefore $w_k = 0$
since otherwise the absolute value of the last summand would dominate
the rest for sufficiently large $\ell$. Similarly, $w_{k-1},\dots,w_1 = 0$.
This proves $T_pT_X = T_pC$ implying $T_X = C$.
\end{proof}

\begin{corollary}[cf.\ \cite{Take-65, LoosCubic, EQT}, see Remark \ref{REM: HC}]
Maximal tori of compact extrinsically symmetric spaces have rectangular unit lattices.
\end{corollary}

Combining Proposition \ref{PROP: Clifford type implies extrinsic symmetry} with Theorem \ref{THM: Maximal tori of ess are Clifford}
and observing that any flat torus contains a dense geodesic we get:

\begin{corollary}\label{COR: Summary}
Let $X$ be a connected submanifold of a Euclidean space $E.$ Then the following conditions are equivalent:
\begin{enumerate}[(i)]
 \item $X$ is a compact extrinsically symmetric space.
 \item $X$ is of Clifford type.
\item $X$ is intrinsically a compact symmetric space and every maximal torus of $X$ a Clifford torus in $E.$
\item $X$ is intrinsically a compact symmetric space and every maximal torus of $X$ is an extrinsically symmetric space in $E.$
\end{enumerate}
\end{corollary}

\begin{remark}
Corollary \ref{COR: Summary} is closely related to the results of \textsc{Ferus} and \textsc{Schirr\-macher} in
\cite{Fe-Sch-82}. However, it seems that they take the implication `(i)$\implies$(iv)' for granted.
\end{remark}

\section{Harish-Chandra's theorem on strongly orthogonal roots}
\label{SECT: Harish Chandra}

We give a new proof of \textsc{Harish-Chandra}'s result on the existence of strongly orthogonal roots by means of
Theorem \ref{THM: Maximal tori of ess are Clifford} above.\par

Let $G$ be a compact, connected, semisimple Lie group with a bi-invariant metric, $\lieG$ its Lie algebra and $\xi\in\lieG$
an non-zero element with $\ad(\xi)^3=-\ad(\xi).$ Then
$X:=\Ad(G)\xi$ is a hermitian symmetric space of compact type
and $X\subset \lieG$ is extrinsically symmetric.
In fact, the 1-parameter group $\varphi_t=e^{t\ad(\xi)}=\Ad(\exp(t\xi))$ of isometries of $\lieG$ leaves $X$ invariant and $\sigma:=\varphi_{\pi}$ is the reflection along the affine normal space (in fact a linear subspace) of $X$ at $\xi.$ Indeed,
$\lieK:=\lieZ(\xi)=\mathrm{ker}(\ad(\xi))$ and $\lieP=[\lieG,\xi]$ are the normal and the tangent spaces of $X$ at $\xi,$
and $\varphi_t=\mathrm{id}$ on $\lieK$ and $\varphi_t=\cos(t)\cdot \mathrm{id}+\sin(t)\cdot \ad(\xi)$ on $\lieP$ due to $\ad(\xi)^3=-\ad(\xi).$ Thus $\sigma$ is an involution with fixed point set $\lieK$ and $(-1)$-eigenspace $\lieP.$ Hence $X\cong G/G_{\xi}$
with $G_{\xi}=\{g\in G:\; \Ad(g)\xi=\xi\}$ is a symmetric space of compact type which moreover is hermitian symmetric as $J:=\varphi_{\frac{\pi}{2}}|_{\lieP}$ defines an $\Ad(G_\xi)$-invariant almost complex structure on $\lieP.$
Conversely, every hermitian symmetric space of compact type is obtained
in this way (cf.\ \cite[Ch.\ VIII, Thm.\ 4.5(i)]{Helg-78} or cf.\ \cite[proof of Thm.\ 4.9, pp.\ 240--243]{Ise-Takeuchi}).\par

Let $\lieH$ be a Cartan (that is maximal abelian) subalgebra of $\lieG$ containing $\xi$ and $\lieG^\C=\lieH^\C\oplus
\sum\limits_{\alpha\in\Root}\lieG_\alpha$ the corresponding root space decomposition of the complexification of
$\lieG$ where an element $\alpha\in\lieH^*\setminus\{0\}$ lies in $\Root$ if and only if
$\lieG_\alpha:=\{X\in\lieG^{\C}: \ad(H)X=i\alpha(H)X \; \text{for all $H\in\lieH$}\}\neq \{0\}.$
Note that $\ad(\xi)^3=-\ad(\xi)$ implies $\alpha^3(\xi)=\alpha(\xi)$ for all $\alpha\in\Root.$
The roots $\alpha\in\R$ with $\alpha(\xi)\neq 0$ are called \emph{of noncompact type}
(this name has its origin in the dual case $\lieG^*=\lieK\oplus i\lieP$) and two roots $\alpha,\beta\in\R$
are called \emph{strongly orthogonal} if $\alpha\pm\beta\notin\Root\cup\{0\}.$

\begin{theorem}[\textsc{Harish-Chandra} {\cite[p.\ 582 f]{HC-56}}]
\label{THM: H-C}
There exists $r=\mathrm{rank}(X)$ pairwise strongly orthogonal roots of noncompact type.
\end{theorem}

\begin{proof}
Let $\lieA\subset\lieP$ be a maximal abelian subspace of $\lieP$ and $\lieH'$ a Cartan subalgebra of $\lieG$ containing $\lieA.$
Then $r=\mathrm{dim}(\lieA)$ and $\lieH'=\lieH_{\lieK}\oplus\lieA$ where $\lieH_{\lieK}=\lieH'\cap \lieK.$ Note however that $\lieH'$ does not contain $\xi.$\par
Since $X$ is extrinsically symmetric and thus of Clifford type by Theorem \ref{THM: Maximal tori of ess are Clifford}, there exists pairwise
orthogonal vectors $X_1,\dots, X_r\in\lieA$ such that the geodesics $\gamma_j(t)=\Ad(\exp(t X_j)\xi$, $j=1,\dots, r,$ in $X$ are
planar circles of smallest period $2\pi$ in $\lieG,$ contained in pairwise orthogonal planes. In particular, $\gamma'''_j(0)=-\gamma'_j(0).$
Let $Y_j=\gamma_j'(0)$ and $H_j=\gamma_j''(0).$ Then $Y_j=[X_j,\xi]\in\lieP$, $H_j=[X_j,Y_j]\in\lieK=\lieZ(\xi)$ and
$Y_1,\dots, Y_r,H_1,\dots, H_r$ are pairwise orthogonal. Thus for each $j$ we get a 3-dimensional subalgebra
$\lieG_j:=\mathrm{span}\{X_j,Y_j,H_j\}$ of $\lieG$ as $[X_j,H_j]=\gamma_j'''(0)=-\gamma_j'(0)=-Y_j$ and
$[Y_j,H_j]=[[X_j,\xi],H_j]=[[X_j,H_j],\xi]=-[Y_j,\xi]=-\ad(\xi)^2X_j=X_j.$
Moreover, these subalgebras pairwise commute as $\lieA$ (containing all $X_j$) and $[\lieA,\xi]=\varphi_{\frac{\pi}{2}}(\lieA)$ (containing all $Y_j$) are abelian and
$\|[X_j,Y_k]\|^2=\langle   [X_j,[X_k,\xi]],[X_j,Y_k]\rangle
=\langle   [X_k,[X_j,\xi]],[X_j,Y_k]\rangle
=-\langle Y_j, [X_k,[X_j,Y_k]]\rangle
=-\langle Y_j, [X_j,[X_k,Y_k]]\rangle
=-\langle Y_j, [X_j,H_k]\rangle
=\langle H_j,H_k\rangle=0$ for $j\neq k$ implying the remaining relations by the Jacobi identity.\par
Let $\lieH:=\lieH_\lieK+\mathrm{Span}\{H_1,\dots, H_r\}.$ Since $\lieH_{\lieK}$ commutes with $\lieA$ and $\xi,$
and $H_j=[X_j,[X_j,\xi]]$, $\lieH$ is abelian and $\lieH_{\lieK}$ is orthogonal to $H_1,\dots, H_k.$ In particular
$\mathrm{dim}(\lieH)=\mathrm{dim}(\lieH').$ Therefore $\lieH$ is a Cartan subalgebra of $\lieH$ contained in $\lieK$ which implies
$\xi\in\lieH.$ For $j\in\{1,\dots, r\}$ we define $\alpha_j:=\frac{1}{\|H_j\|^2}\langle H_j, \;\; \rangle \in\lieH^*.$
Then $\alpha_1,\dots, \alpha_r$ are pairwise strongly orthogonal roots of $\lieG$ of noncompact type as
$\alpha_j(\xi)=-\frac{\|Y_j\|^2}{ \|H_j\|^2}$, and $\lieG_{\alpha_j}$ is spanned by $X_j+iY_j,$ and
$\lieG_{\alpha_j\pm\alpha_k}=[\lieG_{\alpha_j},\lieG_{\pm\alpha_k}]=0$ for  $j\neq k.$
\end{proof}

\begin{remark}\label{REM: HC}
 Here we indicate how conversely  Theorem \ref{THM: Maximal tori of ess are Clifford}  may be deduced from \textsc{Harish-Chandra}'s Theorem \ref{THM: H-C} using results of \textsc{Ferus} and \textsc{Takeuchi}. Let $X=\Ad(G)\xi$ be the extrinsically symmetric submanifold of $\lieG$ as above. If
 $\alpha_1,\dots, \alpha_r$ are $r=\mathrm{rank}(X)$ pairwise strongly orthogonal noncompact roots with respect to a Cartan subalgebra of $\lieG$ containing $\xi$ and if $\lieG_{\alpha_j}=\C(X_j+iY_j),$ then
 the $\lieG_j:=\mathrm{span}\{X_j,Y_j, H_j:=[X_j,Y_j]\}$, $j=1,\dots, r$, are pairwise commuting subalgebras of $\lieG$ isomorphic to $\lieSO_3(\R).$ The orbits through $\xi$ of the corresponding subgroups of $\Ad(G)$ are round 2-spheres as they are complex and thus orientable submanifolds. Their product is called a \emph{polysphere} and is not only a totally geodesic submanifold of $X$ (yielding \textsc{Harish-Chandra}'s polysphere theorem, see \cite{Wolf-Hermitian}), but also extrinsically a product of round 2-spheres in pairwise orthogonal $3$-spaces. Therefore the product of great circles passing through $\xi$ in each of these 2-spheres gives a totally geodesic Clifford torus in $X$ and thus proves Theorem \ref{THM: Maximal tori of ess are Clifford} in this special case.\par
 According to \textsc{Ferus} \cite{Feru-80} (cf.\ \cite{EH-95} for a different proof) every compact, connected, extrinsically symmetric space is congruent to an orbit $Y:=\Ad_G(K')\xi$ in $\lieP'.$ Here $(G,K')$ is a symmetric pair with $G$ compact and semisimple, $\lieG=\lieK'\oplus\lieP'$ is the
 corresponding decomposition of the Lie algebra of $G$ and $\xi$ is an element of $\lieP'$ with $\ad(\xi)^3=-\ad(\xi).$
 In particular, $Y$ is contained in $X=\Ad(G)\xi.$
 \textsc{Takeuchi} \cite{Take-65} has noticed that $Y$ is
 a connected component of the fixed point set of $-\tau|_X:X\to X,$ where $\tau$ is the involution of $\lieG$ with eigenspaces $\lieK'$
 and $\lieP'.$ Thus Theorem \ref{THM: Maximal tori of ess are Clifford} follows for all compact extrinsically symmetric spaces by Proposition \ref{PROP: fixed sets isometries Clifford type}.\par
  In \cite[Thm.\ 1, p.\ 158 and Thm.\ 5, p.\ 164]{Take-65}  \textsc{Takeuchi} has used the ideas sketched above
  to prove that symmetric $R$-spaces have rectangular unit lattices,
  an essential step for the proof of Theorem \ref{THM: Maximal tori of ess are Clifford}.
\end{remark}



\end{document}